\documentclass[12pt]{amsart}
\usepackage{amsmath, amsthm, amssymb}
\usepackage[top=1.25in, bottom=1.25in, left=1.0in, right=1.0in]{geometry}

\makeatletter
\newtheorem*{rep@theorem}{\rep@title}
\newcommand{\newreptheorem}[2]{
\newenvironment{rep#1}[1]{
 \def\rep@title{#2 \ref{##1}}
 \begin{rep@theorem}}
 {\end{rep@theorem}}}
\makeatother

\theoremstyle{plain}
\newtheorem{thm}{Theorem}
\newreptheorem{thm}{Theorem}
\newtheorem*{Brooks}{Brooks' Theorem}
\newtheorem*{KernelLemma}{Kernel Lemma}

\newreptheorem{prop}{Proposition}
\newtheorem{lem}[thm]{Lemma}
\newreptheorem{lem}{Lemma}

\newreptheorem{conjecture}{Conjecture}

\newreptheorem{cor}{Corollary}

\theoremstyle{definition}

\theoremstyle{remark}

\title{A different short proof of Brooks' theorem}
\author{Landon Rabern}

\newcommand{\IN}{\mathbb{N}}

\newcommand{\set}[1]{\left\{ #1 \right\}}

\newcommand{\card}[1]{\left|#1\right|}
\newcommand{\size}[1]{\left\Vert#1\right\Vert}

\newcommand{\func}[3]{#1\colon #2 \rightarrow #3}

\newcommand{\DefinedAs}{\mathrel{\mathop:}=}

\begin{document}
\begin{abstract}
Lov\'{a}sz gave a short proof of Brooks' theorem by coloring greedily in a good
order. We give a different short proof by reducing to the cubic case.  Then we show how to extend the result to (online) list coloring via the Kernel Lemma.
\end{abstract}
\maketitle

In \cite{Lovasz1975269} Lov\'{a}sz gave a short proof of Brooks' Theorem \cite{brooks1941colouring} by
coloring greedily in a good order. Here we give a different short proof by reducing to the cubic case.  One
interesting feature of the proof is that it doesn't use any connectivity
concepts. Our notation follows Diestel \cite{Diestel} except we write $K_t$ instead of $K^t$ for the complete graph on $t$ vertices.  We'll give two versions of the proof, the first is shorter but uses the extra idea of excluding diamonds ($K_4$ less an edge).

\begin{Brooks}
Every graph satisfies $\chi \leq \max\set{3, \omega, \Delta}$.
\end{Brooks}
\begin{proof}
Suppose the theorem is false and choose a counterexample $G$ minimizing
$\card{G}$.  Put $\Delta \DefinedAs \Delta(G)$. Using minimality of $\card{G}$,
we see that $\chi(G - v) \leq \Delta$ for all $v \in
V(G)$. In particular, $G$ is $\Delta$-regular.

First, suppose $G$ is $3$-regular.  If $G$ contains a diamond $D$, then we may $3$-color $G-D$ and easily extend the coloring to $D$ by first coloring the nonadjacent vertices in $D$ the same.  So, $G$ doesn't contain diamonds. Since $G$ is not a forest it contains an induced cycle $C$. Since $K_4 \not
\subseteq G$ we have $\card{N(C)} \geq 2$. So, we may take different $x, y \in N(C)$ and put $H \DefinedAs G - C$ if $x$ is adjacent to $y$ and $H \DefinedAs (G-C) + xy$ otherwise.  Then, $H$ doesn't contain $K_4$ as $G$ doesn't contain diamonds. By minimality of $\card{G}$, $H$ is $3$-colorable. That is, we have a $3$-coloring of $G - C$ where $x$ and $y$ receive different colors.  We can easily extend this partial
coloring to all of $G$ since each vertex of $C$ has a set of two available
colors and some pair of vertices in $C$ get different sets.  

Hence we must have $\Delta \geq 4$. Consider a $\Delta$-coloring of $G-v$ for some $v \in V(G)$.  Each color must be used on every $K_{\Delta}$ in $G-v$ and hence some color must be used on every $K_{\Delta}$ in $G$.  Let $M$ be such a color class expanded to a maximal independent set.  Then $\chi(G-M) = \chi(G) - 1 = \Delta > \max\set{3, \omega(G-M), \Delta(G-M)}$, a contradiction.
\end{proof}

Here is another version of the proof, not excluding diamonds and doing the reduction differently.

\begin{Brooks}
Every graph $G$ with $\chi(G) = \Delta(G) + 1 \geq 4$ contains
$K_{\Delta(G) + 1}$.
\end{Brooks}
\begin{proof}
Suppose the theorem is false and choose a counterexample $G$ minimizing
$\card{G}$.  Put $\Delta \DefinedAs \Delta(G)$. Using minimality of $\card{G}$,
we see that $\chi(G - v) \leq \Delta$ for all $v \in
V(G)$. In particular, $G$ is $\Delta$-regular.

First, suppose $\Delta \geq 4$.  Pick $v \in V(G)$ and let $w_1, \ldots,
w_\Delta$ be $v$'s neighbors. Since $K_{\Delta + 1} \not \subseteq G$, by
symmetry we may assume that $w_2$ and $w_3$ are not adjacent. Choose a $(\Delta
+ 1)$-coloring $\set{\set{v}, C_1, \ldots, C_\Delta}$ of $G$ where $w_i \in
C_i$ so as to maximize $\card{C_1}$.  Then $C_1$ is a maximal independent set in
$G$ and in particular, with $H \DefinedAs G - C_1$, we have $\chi(H) =
\chi(G) - 1 = \Delta = \Delta(H) + 1 \geq 4$.  By minimality of $\card{G}$, we
get $K_\Delta \subseteq H$.  But $\set{\set{v}, C_2, \ldots, C_\Delta}$ is a
$\Delta$-coloring of $H$, so any $K_\Delta$ in $H$ must contain $v$ and hence
$w_2$ and $w_3$, a contradiction.

Therefore $G$ is $3$-regular.  Since $G$ is not a forest it contains an induced
cycle $C$.  Put $T \DefinedAs N(C)$.  Then $\card{T} \geq 2$ since $K_4 \not
\subseteq G$.  Take different $x, y \in T$ and put $H_{xy} \DefinedAs G - C$ if
$x$ is adjacent to $y$ and $H_{xy} \DefinedAs (G-C) + xy$ otherwise.  Then, by
minimality of $\card{G}$, either $H_{xy}$ is $3$-colorable or adding $xy$
created a $K_4$ in $H_{xy}$.

Suppose the former happens.  Then we have a $3$-coloring of $G - C$
where $x$ and $y$ receive different colors.  We can easily extend this partial
coloring to all of $G$ since each vertex of $C$ has a set of two available
colors and some pair of vertices in $C$ get different sets. 

Whence adding $xy$ created a $K_4$, call it $A$, in $H_{xy}$.  We conclude that
$T$ is independent and each vertex in $T$ has exactly one neighbor in $C$.  Hence
$\card{T} \geq \card{C} \geq 3$. Pick $z \in T - \set{x,y}$.  Then $x$ is
contained in a $K_4$, call it $B$, in $H_{xz}$.  Since $d(x) = 3$, we must have
$A - \set{x,y} = B - \set{x, z}$.  But then any $w \in A - \set{x,y}$ has degree
at least $4$, a contradiction.
\end{proof}

We note that the reduction to the cubic case is an immediate consequence of more
general lemmas on hitting all maximum cliques with an independent set
(see \cite{kostochkaRussian}, \cite{rabernhitting} and \cite{KingHitting}).  H.
Tverberg pointed out that this reduction was also demonstrated in his paper
\cite{tverberg1983brooks}.

\section*{Lifting to (online) list coloring}

In \cite{kostochkayancey2012ore}, Kostochka and Yancey gave a simple, yet powerful, application of the Kernel Lemma to a list coloring problem.  In \cite{orevizing}, we generalized
their idea to get a widely applicable general tool.  As an illustration, we use a special case of this tool to prove the list coloring version of Brooks' Theorem (first proved by Vizing \cite{vizing1976}).
The same argument proves the online list coloring version of Brooks' Theorem, but we will stick to ordinary list coloring for simplicity.

A \emph{kernel} in a digraph $D$ is an independent set $I \subseteq V(D)$ such that each vertex in $V(D) - I$ has an edge into $I$.  A digraph in which every induced subdigraph has a kernel is called \emph{kernel-perfect}.
A simple induction argument shows that kernel-perfect orientations can be very useful for list coloring:

\begin{KernelLemma}
Let $G$ be a graph and $\func{f}{V(G)}{\IN}$. If $G$ has a kernel-perfect orientation such that $f(v) \geq d^+(v) + 1$ for each $v \in V(G)$, then $G$ is $f$-choosable.
\end{KernelLemma}

All bipartite graphs are kernel-perfect, the following lemma from \cite{kostochkayancey2012ore} generalizes this fact.

\begin{lem}\label{KernelPerfect}
Let $A$ be an independent set in a graph $G$ and put $B \DefinedAs V(G) - A$.  Any digraph created from $G$ by replacing each edge in $G[B]$ by a pair of opposite arcs and orienting the edges between $A$ and $B$ arbitrarily is kernel-perfect.
\end{lem}
\begin{proof}
In a minimum counterexample $G$, to get a contradiction it suffices to construct a kernel.  But this is easy since either $A$ is a kernel or there is some $v \in B$ which has no out neighbors in $A$.  In the latter case, a kernel in $G - v - N(v)$ together with $v$ is a kernel in $G$.
\end{proof}

\begin{thm}
Every graph satisfies $\chi_l \le \max\set{3, \omega, \Delta}$.
\end{thm}
\begin{proof}
Suppose the theorem is false and choose a counterexample $G$ minimizing
$\card{G}$.  Put $\Delta \DefinedAs \Delta(G)$. Using minimality of $\card{G}$,
we see that $\chi_l(G - v) \leq \Delta$ for all $v \in
V(G)$. In particular, $G$ is $\Delta$-regular.

By Brooks' Theorem, $\chi(G) \le \Delta$ and hence $\alpha(G) \ge \frac{|G|}{\Delta}$.  Let $A$ be a maximum independent set in $G$ and put $B \DefinedAs V(G) - A$.
Then there are $\alpha(G)\Delta \ge |G|$ edges between $A$ and $B$.  Hence there exist nonempty induced subgraphs $H$ of $G$ with at least $|H|$ edges between $A \cap V(H)$ and $B \cap V(H)$. 
Pick such an $H$ minimizing $|H|$.  For $v \in V(H)$, let $d_v$ be the number of edges between $A \cap V(H)$ and $B \cap V(H)$ incident to $v$. We'll use minimality of $|H|$ to show that $d_v=2$ for all $v \in V(H)$.  Clearly, $d_v \ge 2$.  Suppose $d_v < d_w$ for some other $w \in V(H)$. Then there are at least $|H| + d_w - 2$ edges across and hence removing $v$ violates minimality.  So, $d_v = d_w$ for all $v,w \in V(H)$. Call this common value $d$.  Then there are at least $(d-1)|H|$ edges across.  Since $(d-1)|H| - d \ge |H| - 1$ for $d > 2$, applying minimality shows that $d=2$.  Hence the edges between $A \cap V(H)$ and $B \cap V(H)$ induce a disjoint union of cycles.

Create a digraph $Q$ from $H$ by replacing each edge in $H[B \cap V(H)]$ by a pair of opposite arcs and orienting the edges between $A \cap V(H)$ and $B \cap V(H)$ around their respective cycles.  Then $Q$ is kernel-perfect by Lemma \ref{KernelPerfect}.  Since $d^+(v) \le d(v) - 1$ for each $v \in V(Q)$, applying the Kernel Lemma shows that $H$ is $f$-choosable where $f(v) \DefinedAs d(v)$ for all $v \in V(H)$.  But then given any $\Delta$-list-assignment on $G$, we can color $G-H$ from its lists by minimality of $|G|$ and then complete the coloring to $H$, a contradiction.
\end{proof}

The ad-hoc orientation construction in the above proof can be replaced with the following general lemma.  This lemma follows easily by taking an arbitrary orientation and repeatedly reversing paths if doing so gets a gain (really, this is just the proof of the max-flow min-cut theorem).  This can also be proved by splitting vertices and applying Hall's theorem.

\begin{lem}\label{InOrientations}
Let $G$ be a graph and $\func{g}{V(G)}{\IN}$.  Then $G$ has an orientation such that $d^-(v) \geq g(v)$ for all $v \in V(G)$ iff for every $H \unlhd G$ we have

\[\size{H} + \size{H, G-H} \ge \sum_{v \in V(H)} g(v).\]
\end{lem}

Using Lemma \ref{InOrientations} in a similar way to the proof of Brooks' Theorem for list coloring above, we get the general tool from \cite{orevizing}.

\begin{lem}\label{SecondStrengtheningMic}
Let $G$ be a nonempty graph and $\func{f}{V(G)}{\IN}$ with $f(v) \leq d_G(v) + 1$ for all $v \in V(G)$. If there is independent $A \subseteq V(G)$ such that

\[\size{A, G-A} \ge  \sum_{v \in V(G)} d_G(v) + 1 - f(v),\]

\noindent then $G$ has a nonempty induced subgraph $H$ that is (online) $f_H$-choosable where $f_H(v) \DefinedAs f(v) + d_H(v) - d_G(v)$ for $v \in V(H)$.
\end{lem}

In the case of Brooks' theorem, we want $f(v) \DefinedAs d_G(v)$, so the right side of the condition in Lemma \ref{SecondStrengtheningMic} just equals $|G|$ and thus the existence of the desired $A$ follows immediately from
$\alpha(G) \ge \frac{|G|}{\Delta}$.  More generally, as shown in \cite{orevizing}, the classification of (online) degree-choosable graphs is a quick corollary of Lemma \ref{SecondStrengtheningMic}.

\end{document}